\documentclass[11pt,a4paper]{amsart}
\allowdisplaybreaks[1]
\usepackage{mathtools}
\mathtoolsset{showonlyrefs}
\usepackage{amssymb}
\usepackage{mathrsfs}
\usepackage{amsbsy,amsgen,amscd,amsthm}
\usepackage{bm}
\usepackage{dsfont}
\usepackage{graphicx}
\usepackage{subfig}
\usepackage{caption}
\usepackage{float}
\usepackage{enumerate}
\usepackage{enumitem}
\usepackage{setspace}
\usepackage{tikz}
\usepackage[top=3.2cm,bottom=3.8cm,left=3cm,right=2cm]{geometry}
\usepackage{xcolor}
\usepackage[T1]{fontenc}
\usepackage{lmodern}
\usepackage[unicode=true,colorlinks,linkcolor=blue,citecolor=blue,urlcolor=blue,pagebackref]{hyperref}

\renewcommand {\thefootnote}{\fnsymbol{footnote}}

\theoremstyle{plain}
\newtheorem{theorem}{Theorem}[section]
\newtheorem{lemma}[theorem]{Lemma}

\newtheorem{problem}{Problem}

\theoremstyle{definition}
\newtheorem{definition}[theorem]{Definition}

\theoremstyle{remark}
\newtheorem{remark}[theorem]{Remark}

\numberwithin{equation}{section}

\newcommand{\R}{\mathbb{R}}

\newcommand{\sphere}{\mathbb{S}}
\newcommand{\Haus}{\mathcal{H}}
\newcommand{\Hem}{\operatorname{T}}
\newcommand{\eps}{\varepsilon}
\newcommand{\dd}{\,d}
\newcommand{\ip}[2]{\langle #1,#2\rangle}
\newcommand{\Harm}{\mathscr{H}}
\let\mathbbm\mathds

\newcommand\nnfootnote[1]{%
	\begin{NoHyper}
		\renewcommand\thefootnote{}\footnote{#1}%
		\addtocounter{footnote}{-1}%
	\end{NoHyper}
}

\begin{document}
\enlargethispage{3pt}
	
	\begin{center}
		{\large\bf A complete solution to the Gr\"unbaum--Loewner centroid problems}
	\end{center}
	
	\vskip 15pt
	\begin{center}
		{\small\bf Shouda\ Wang$^{1}$\ \  \ \ \ \ Ge\ Xiong$^{2}$\ \  \ \ \ \ Kaiwen\ Yang$^{2}$}\\~~ \\
		\small{$^{1}$Program in Applied and Computational Mathematics, Princeton University, Princeton, NJ 08544, USA}\\
		\small{$^{2}$School of Mathematical Sciences,  Tongji University, Shanghai 200092, P. R. China}
	\end{center}
	
	\vskip 5pt
	\nnfootnote{E-mail addresses: 1. shoudawang@princeton.edu;\ 2. xiongge@tongji.edu.cn;\ 3. yangkaiwen@tongji.edu.cn.}
	
\begin{center}
\begin{minipage}{13.5cm}
{{\bf Abstract:}
In the 1960s,  Gr\"unbaum and Loewner raised problems concerning the minimum possible number of hyperplane sections of a convex body in $\R^n$ whose centroids  coincide with the centroid of the body itself. By employing  the hemispherical transform and Morse theory, we provide a complete solution to these problems.}
			
			
\end{minipage}
\end{center}
	
\vskip 20pt
\section{\bf Introduction}
\vskip 5pt

Let $K$ be a convex body (i.e., a compact convex set with non-empty interior) in the $n$-dimensional Euclidean space, $\R^n$. The \emph{centroid} (also called \emph{center of mass} or \emph{barycenter}) of $K$ is the point
$$
  c(K)=\frac{1}{\Haus^n(K)}\int_K x\dd\Haus^n(x),$$
where $\Haus^n$ denotes the $n$-dimensional Hausdorff measure.
A \emph{hyperplane section} of $K$ is a set $K\cap H$, where $H$ is a hyperplane in $\R^n$ and $\Haus^{n-1}(K\cap H)>0$. So, the centroid of $K\cap H$ is 
$$c(K\cap H)=\frac{1}{\Haus^{n-1}(K\cap H)}\int_{K\cap H} x\dd\Haus^{n-1}(x).$$

For each  \emph{planar} convex set $K$, it  has at least \emph{three} chords whose centroids (i.e., midpoints) coincide with the centroid $c(K)$ of $K$. Please refer to   Bose~\cite{Bose1935ConvexOval} and Ehrhart~\cite{Ehrhart1955Minkowski}.
In 1961,  Grünbaum  asked whether the same phenomenon persists in arbitrary dimension \cite[p.~41]{Gr2}.

\begin{problem}[Grünbaum]\label{prob1}
 Let $K$ be a convex body in $\R^n$. Do there exist at least $n+1$ hyperplane sections whose centroids coincide with $c(K)$?
\end{problem}

In 1965, Loewner asked for the largest integer that can replace $n+1$ in the Grünbaum problem \cite[Problem~28]{Fenchel1967Convexity}.

\begin{problem}[Loewner]
  For each integer $n\geq 2$, define $\mu(n)$ as the largest integer  with the property that  every convex body
  $K\subseteq \R^n$ admits at least $\mu(n)$ hyperplane sections whose centroids coincide with $c(K)$. What is the value of $\mu(n)$?
\end{problem}

Hence, $\mu(2)\geq 3$, and the Grünbaum problem asks whether $\mu(n)\geq n+1$ always holds. By contrast, 
 that $\mu(n)\geq 1$ can be easily derived for dimensions  $n\geq 3$. See Lemma \ref{lem:lower1} for details.

The discrepancy  between the known lower bound $\mu(n)\geq 1$ and Grünbaum's conjectured lower bound $\mu(n)\geq n+1$ persisted  for a long time, until  Myroshnychenko, Tatarko and Yaskin \cite{MTY} recently \emph{disproved} the Grünbaum conjecture for dimensions $n\geq 5$. Precisely, they \cite{MTY} proved  that $\mu(n)=1$ for all $n\geq 5$ by using non-intersection bodies. Since non-intersection bodies do not exist in dimensions $3$ and $4$ (see Gardner~\cite{Gardner1994BP3} and Zhang~\cite{Zhang1999BP4}), their approach cannot cover dimensions $3$ and $4$.  Despite this, they conjectured that $\mu(3)=\mu(4)=1$ as well \cite[Remark~1]{MTY}.

It is interesting that Huang, Myroshnychenko, Tatarko  and Yaskin~\cite{HMTY26} very recently studied a hyperbolic analogue $\mu_{\mathrm{Hyp}}$ of $\mu$, and proved that $\mu_{\mathrm{Hyp}}(2)=3$ and  $\mu_{\mathrm{Hyp}}(n)=1$ for all $n\geq 3$ by using the existence of hyperbolic analogues of non-intersection bodies established by Yaskin~\cite{Yaskin06}. 

In this article, we prove that $\mu(n)\leq 1$ for all dimensions $n\geq 3$. This, together with established results, provides a \emph{complete} solution to the Grünbaum--Loewner problems.

\begin{theorem}\label{thm1}
For each $n\geq 3$, there exists a smooth strictly convex body $K$ in  $\R^n$ for which exactly one hyperplane section has centroid $c(K)$. Consequently, 
$$\mu(n)=
    \begin{cases}
      3, & n=2, \\
      1, & n\geq 3.
    \end{cases} $$
\end{theorem}

In this article,  a convex body $K$ in $\R^n$ is called \emph{smooth}, if $\partial K$ is a $C^\infty$ submanifold of $\R^n$; it is called \emph{strictly convex}, if for all distinct $x,y\in K$ and every $\lambda\in (0,1)$, one has $\lambda x+(1-\lambda)y\in \operatorname{int} K$.

We briefly outline the key ideas underlying the proof of the main results.

First, for a convex body $K$ in $\R^n$, its \emph{half-space volume function} is $A_K:\sphere^{n-1}\to \R,$ 
$$ A_K(u)=\Haus^n\bigl(K\cap \{x\in\R^n:\ip{x-c(K)}{u}\geq 0\}\bigr), \quad \forall \ u\in\sphere^{n-1},
$$ where $\sphere^{n-1}$ is the unit sphere of $\R^n$. Here and throughout this article,  $\ip{\cdot}{\cdot}$ denotes the standard inner product in $\R^n$. Suppose the centroid $c(K)$ of $K$ is at the origin $o$ of $\R^n$.  Then
$A_K$ has the property that its critical points are precisely the normal vectors to hyperplane sections whose centroids lie at the origin
(Lemma \ref{lemma1}).  Thus, to prove Theorem \ref{thm1}, it suffices  to exhibit a convex body  $K$ for which $A_K$  possesses exactly two antipodal critical points.

Second, we invoke the \emph{hemispherical transform} $\Hem$, which links the half-space volume function $A_K$ to the \emph{radial function} $\rho_K$ of  $K$ by 
$$A_K = \frac1n \Hem(\rho_K^n),$$
where $\rho_K(u)= \sup\{\lambda>0:\lambda u\in K\}$, $u\in\sphere^{n-1}$. We prove that $\Hem$ is an \emph{isomorphism} on the space of smooth odd functions on $\sphere^{n-1}$, and that the \emph{degree-one} spherical-harmonic component of $\Hem f$ vanishes  if and only if the {degree-one} spherical-harmonic component of $ f$ vanishes  (Lemma \ref{lem14}). So, if we  find a smooth odd Morse function $f$ on $\sphere^{n-1}$ with exactly two critical points and no degree-one component, 
then solving $\Hem h=f$ gives a smooth odd function $h$ with no degree-one component. 

Third, for small $\eps \in \R$ define
$$
  K_\eps=\{ru:\ 0\leq r\leq (1+\eps h(u))^{1/(n+1)}, ~ u\in\sphere^{n-1}\}.
$$
The absence of the degree-one component in $h$ ensures that $c(K_\eps)=o$, while the expansion of $A_{K_\eps}$ shows that its zeroth-order term is constant and that its rescaled first-order term converges to $f$ in $C^2$ as $\eps\to 0$.
The stability of Morse functions (Lemma \ref{lem5}) then implies that $A_{K_\eps}$ has the same number of critical points as $f$, which is two. Consequently, 
 $K_\eps$ has exactly one hyperplane section whose centroid agrees with $c(K_\eps)$, as desired.

Finally, it remains to construct such a Morse function  $f$. This is accomplished  by twisting a height function on the unit sphere $\sphere^{n-1}$; the details are given in Lemma \ref{prop2}.

\begin{remark}
Grünbaum~\cite{Gr2} further posed  a weaker question: does every convex body $K\subseteq \R^n$ contain a point $p\in\operatorname{int}K$ that serves as the centroid of at least $n+1$ hyperplane sections of $K$? Problem~\ref{prob1} then asks whether we may take $p=c(K)$ in this statement. Grünbaum claimed a proof of this weaker statement in~\cite[Sec.~6.2]{Grunbaum1963Measures}; however, Pat{\'a}kov{\'a}, Tancer and Wagner~\cite{PTW} found a substantive flaw in his argument. Nonetheless, the same authors~\cite{PTW} showed that the weaker statement does hold in dimension $n=3$. The cases $n\geq 4$ remain open to date.
\end{remark}

This paper is organized as follows.  In Section \ref{sec:pre}, we collect some necessary facts about the half-space volume function, the hemispherical transform, spherical harmonics, and perturbative estimates. We also include a proof of the planar case $\mu(2)=3$ in Section \ref{sec:pre} for completeness. In Section \ref{sec:main}, we first prove that the hemispherical transform is an isomorphism on smooth odd functions and is bounded on $C^k$ functions. We then establish a stability result of Morse functions under $C^2$ perturbations, and lastly prove Theorem~\ref{thm1}.

\textbf{Acknowledgments.} The results in this paper are obtained independently and simultaneously by Shouda Wang (at Princeton University) and by Ge Xiong and Kaiwen Yang (at Tongji University). The first author thanks Vlad Yaskin for his comments on an earlier draft. The research of the last two authors was supported by NSFC No. 12271407.

\vskip 20pt
\section{\bf Preliminaries}
\label{sec:pre}
\vskip 5pt

Let $K$ be a convex body in $\R^n$ and write $c(K)$ for its centroid. 
For $u\in\sphere^{n-1}$, write
$$u^\perp=\{x\in\R^n:\ip{x}{u}=0\}
  \quad\text{and}\quad
  u^\perp+c(K)=\{x+c(K):x\in u^\perp\}.$$
Then $K\cap (u^\perp+c(K))$ is a hyperplane section passing through the centroid $c(K)$. 
For the purpose of studying the Grünbaum--Loewner centroid problems, we naturally restrict our attention to such hyperplane sections.

For a smooth convex body $K$ in $\R^n$,  its radial function $\rho_K$ is in $ C^\infty(\sphere^{n-1})$. If   $K$ is a smooth convex body with positive Gauss curvature everywhere, then $K$ is strictly convex.

\subsection{The half-space volume function and the hemispherical transform}
\ 

A starting point of our approach is the following formula (see, e.g., Haddad, Jim{\'e}nez and Villa~\cite[Lemma~8]{HJV25}, and Pat{\'a}kov{\'a}, Tancer and Wagner~\cite[Prop.~11]{PTW}) for the half-space volume function
$ A_K(u)=\Haus^n\bigl(K\cap \{x\in\R^n:\ip{x-c(K)}{u}\geq 0\}\bigr), \  \forall \ u\in\sphere^{n-1}.$

\begin{lemma}\label{lemma1}
If $K\subseteq \R^n$ is a convex body with its centroid $c(K)$ at the origin $o$, then $A_K$ is of class $C^1$ and its directional derivative in $v\in T_u\sphere^{n-1}=u^\perp$ at $u\in \sphere^{n-1}$ is
$$D_v A_K(u)=\int_{K\cap u^\perp}\ip{x}{v}\dd\Haus^{n-1}(x).$$
Consequently, $u$ is a critical point of $A_K$ if and only if
$c(K\cap u^\perp)=c(K)=o$.
\end{lemma}

Since $A_K(-u)=\Haus^n(K)-A_K(u),$  $\forall \ u  \in\sphere^{n-1}$,
it follows that the critical points of $A_K$ come in antipodal pairs.

The well-known bound $\mu(n)\geq 1$ is an immediate
consequence of 
Lemma \ref{lemma1}  (see Gr{\"u}nbaum~\cite[Theorem~1]{Gr2}, Meyer and Reisner~\cite[Lemma~8]{MR89},  or Steinhaus~\cite[(1)]{Steinhaus1955Applications}).

\begin{lemma}\label{lem:lower1}
  If $n\geq 2$ and $K\subseteq \R^n$ is  a convex body with centroid $c(K)=o$,
  then there exists a direction  $u_0\in \sphere^{n-1}$ such that  
  $ c\bigl(K\cap u_0^\perp\bigr)
  =o.$  Hence, $\mu(n)\geq 1$.
\end{lemma}

\begin{proof}
  By Lemma~\ref{lemma1}, the function $A_K$ is of class $C^1$. Since
  $\sphere^{n-1}$ is compact and has no boundary, $A_K$ attains its maximum at some point $u_0\in \sphere^{n-1}$ and 
  $\nabla_{\sphere^{n-1}} A_K(u_0)=0$.
  By Lemma~\ref{lemma1} again, we have $c(K\cap u_0^\perp)=o=c(K)$. In other words, $\mu(n)\geq 1$.
\end{proof}

In the following, we  express the half-space volume function $A_K$ of convex body $K$ as the \emph{hemispherical transform} of its radial function $ \rho_K$, which was introduced by Funk~\cite{Funk1916} on $\sphere^2$ and generalized to arbitrary dimensions by Rubin~\cite{Rubin}. Let $\mathbbm 1$ denote the indicator function. 

\begin{definition}
  The \emph{hemispherical transform} of a function $f\in C(\sphere^{n-1})$ is defined by
$$
  \Hem f(u)
  =
  \int_{\sphere^{n-1}}
  \mathbbm 1_{\{ \ip{\theta}{u}\geq 0\}}
  f(\theta)\dd\Haus^{n-1}(\theta),\quad  \forall \ u\in \sphere^{n-1}.
$$
\end{definition}

Let $K$ be a convex body in $\R^n$ with $c(K)=o$. Then, for every $u\in \sphere^{n-1}$,
\begin{align}\nonumber
  A_K(u)
  &= \Haus^n\bigl(K\cap \{x\in\R^n:\ip{x}{u}\geq 0\}\bigr) \\
  \nonumber
  &= \int_{\sphere^{n-1}}
  \int_0^{\rho_K(\theta)}\mathbbm 1_{\{\ip{\theta}{u}\geq 0\}} r^{n-1}\dd r\dd\Haus^{n-1}(\theta) \\
  &= \frac1n
  \int_{\sphere^{n-1}}\mathbbm 1_{\{\ip{\theta}{u}\geq 0\}}
  \rho_K(\theta)^n\dd\Haus^{n-1}(\theta) \nonumber.
 \end{align}
That is,
\begin{align}
 A_K(u)=\frac1n\Hem(\rho_K^n)(u),\quad \forall \ u\in \sphere^{n-1}. \label{eq3}
\end{align}
Formula~\eqref{eq3} is the bridge between the geometry of hyperplane sections and our analytic construction: it allows us to control the critical points of $A_K$ by choosing the radial function $\rho_K$ through the hemispherical transform.

\subsection{Spherical harmonics}
\ 

For later use, we collect some basic facts on spherical harmonics. One can refer to Groemer~\cite[Section 3]{Groemer}, Schneider~\cite[Appendix]{Schneider}, and Dai and Xu~\cite{DaiXu} for more details.

A \emph{spherical harmonic} of degree $m$ on $\sphere^{n-1}$ is the restriction to $\sphere^{n-1}$ of a homogeneous harmonic polynomial of degree $m$ on $\R^n$. If $Y_m$ is a spherical harmonic of degree $m$, then 
$$
  Y_m(-u)=(-1)^mY_m(u)
  \quad \text{and}\quad
  \Delta_{\sphere^{n-1}}Y_m=-m(m+n-2)Y_m.
$$
For $n \geq 3$, the set $\Harm^m$ of spherical harmonics of degree $m$ is a finite-dimensional vector subspace of $C^\infty(\sphere^{n-1})$ (and hence of $L^2(\sphere^{n-1})$), with dimension
$$
	\dim \mathscr{H}^m = N(n,m) = \frac{(2m+n-2)\Gamma(n+m-2)}{\Gamma(m+1)\Gamma(n-1)}.
	$$
For $n=2$, $\dim \mathscr{H}^0=1$ and $\dim \mathscr{H}^m=2$ for $m \geq 1$.

Assume the spherical harmonics are real-valued. 
For each $m \geq 0$, we choose an orthonormal basis $\{Y_{m,k}\}_{k=1}^{N(n,m)}$ of $\Harm^m$ with respect to the $L^2$ inner product on $\sphere^{n-1}$ defined by
$$
 (f,g):=\int_{\sphere^{n-1}}f(u)g(u)\dd\Haus^{n-1}(u),\quad  \forall \ f,g\in L^2(\sphere^{n-1}).
$$
Spherical harmonics of different degrees are orthogonal, i.e.,
 $$(f, g) = 0,  \quad  \forall \  f \in \mathscr{H}^k,\  g \in \mathscr{H}^j,\  k \neq j,$$
 and the algebraic direct sum $\bigoplus_{m=0}^{\infty}\Harm^m$ is dense in $L^2(\sphere^{n-1})$. Consequently, the collection
 $$\{Y_{m,k} : m = 0,\ 1,\ 2,\ldots,\ k = 1, \ldots, N(n,m)\}$$ forms a complete orthonormal system of $L^2(\sphere^{n-1})$.

Let $\pi_m:L^2(\sphere^{n-1})\to \Harm^m$ denote the orthogonal projection onto $\Harm^m$. With the chosen orthonormal basis, for every $f\in L^2(\sphere^{n-1})$, it is explicitly given by
 $$
     \pi_m f=\sum_{k=1}^{N(n,m)}(f,Y_{m,k})Y_{m,k} \quad\text{and} \quad f=\sum_{m=0}^\infty \pi_m f \quad \text{in }L^2(\sphere^{n-1}). $$
Since spherical harmonics of odd (resp. even) degrees are odd (resp. even) functions,
it follows that $f$ is odd if and only if $\pi_m f=0$ for each even $m$.

In particular, the space $\Harm^1$ consists precisely of the restrictions to $\sphere^{n-1}$ of all linear functions on $\R^n$.
The coordinate functions $u_1,\dots,u_n$, where $(u_1,\dots,u_n)\in \sphere^{n-1}$, form an orthogonal basis of $\Harm^1$.
Thus, $f\in L^2(\sphere^{n-1})$ has \emph{no degree-one component}, i.e., its orthogonal projection $\pi_1 f=0$, if and only if $f$ is orthogonal to each $u_i$, which holds if and only if 
\begin{equation}\label{eq:no-H1}
  \int_{\sphere^{n-1}}u\,f(u)\dd\Haus^{n-1}(u)=0.
\end{equation}

\subsection{Perturbative estimates}
\ 

The following lemma states that the class of smooth convex bodies with positive Gauss curvature everywhere is stable under sufficiently small $C^2$ perturbation of the radial functions.
\begin{lemma}\label{lem3}
  If $\rho_0\in C^\infty(\sphere^{n-1})$ is a radial function of a smooth convex body with positive Gauss curvature everywhere,
  then there exists $\eps>0$ such that for all
  $\rho\in C^\infty(\sphere^{n-1})$ with $\|\rho-\rho_0\|_{C^2(\sphere^{n-1})}<\eps$,
  the set
  $K_\rho=\{ru:\ 0\leq r\leq \rho(u),\ u\in\sphere^{n-1} \}$
  is also a  smooth convex body with positive Gauss curvature everywhere. 
\end{lemma}

\begin{proof}
This is a standard perturbation argument. See, e.g., Koldobsky~\cite[p.~96]{Koldobsky2005} or Myroshnychenko, Tatarko and Yaskin~\cite[Lemma~4]{MTY}.

  Let $\sigma_{ij}$ be the standard metric on $\sphere^{n-1}$, and let
  $\nabla$ be its Levi--Civita connection. In local coordinates, the second
  fundamental form of $\partial K_\rho$ is
  $$ \mathrm{II}^{\rho}_{ij}=
    \frac{
      \rho^2\sigma_{ij}
      +2\nabla_i\rho\,\nabla_j\rho
      -\rho\nabla_{ij}\rho}{\sqrt{\rho^2+\lvert\nabla\rho\rvert^2} }.$$
  See Guan and Spruck~\cite[\S2, equation (2.3)]{GuanSpruck1993}. 
  
  For $\rho=\rho_0$, $\mathrm{II}^{\rho_0}$ is positive definite since $\rho_0$ is the radial function of a smooth convex body with positive Gauss curvature everywhere. Since $\sphere^{n-1}$ is compact and $\mathrm{II}^{\rho}$ depends continuously on $\rho$ in the $C^2$-topology, $\mathrm{II}^{\rho}$ is positive definite for all $\rho$ with $\|\rho-\rho_0\|_{C^2(\sphere^{n-1})}<\eps$ for sufficiently small $\eps>0$. Thus, $K_\rho$ is a smooth convex body with positive Gauss curvature everywhere.
\end{proof}

\begin{lemma}\label{lem4}
 If $h\in C^\infty(\sphere^{n-1})$ and $\beta>0$, then
  $ \|(1+\eps h)^\beta-1-\beta\eps h\|_{C^2(\sphere^{n-1})} = O(\eps^2),~\text{as }\ \eps\to 0.$
\end{lemma}

\begin{proof}
  Let
  $$\varphi(s)=(1+s)^\beta-1-\beta s,\quad s>-1.$$
  Since $\varphi(0) = \varphi'(0)=0$, there exist $\delta>0$ and a smooth function $\psi \in C^\infty(-\delta,\delta)$ such that $\varphi(s)=s^2\psi(s)$ for $|s|<\delta$. Let $M_0=\|h\|_{C^0(\sphere^{n-1})}$. If $M_0=0$, the claim is trivial. Otherwise, choose $\eps_0=\frac{\delta}{2M_0}$. Then for all $|\eps|\leq \eps_0$,  we have $|\eps h(u)|\leq \frac{\delta}{2},$ for  $ u\in \mathbb{S}^{n-1}$. Hence, the function
  $$ (1+\eps h)^\beta-1-\beta\eps h
    =\eps^2 h^2\psi(\eps h)$$
  is well defined on $\mathbb{S}^{n-1}$.

In the following,  we show that $\|h^2\psi(\varepsilon h)\|_{C^2(\mathbb S^{n-1})}$ is bounded uniformly for $|\varepsilon|\le\varepsilon_0$. Let
  $$
    A_j=\sup_{|t|\leq\delta/2}|\psi^{(j)}(t)|,
    \qquad j=0,1,2.
  $$
Then $A_0,A_1,A_2<\infty$ by $\psi \in C^\infty(-\delta,\delta)$. For brevity, write $F_\varepsilon=h^2\psi(\varepsilon h)$. Then
  $$
    |F_\varepsilon|\leq A_0\|h\|_{C^0(\sphere^{n-1})}^2=A_0M_0^2.
  $$
  
  Since
  $$
    \nabla F_\varepsilon
    =
    \bigl(2h\psi(\varepsilon h)
    +\varepsilon h^2\psi'(\varepsilon h)\bigr)\nabla h,
  $$
   $h\in C^\infty(\mathbb S^{n-1})$ and
  $|\varepsilon h|\leq\frac{\delta}{2}$, it follows that $\|\nabla F_\varepsilon\|_{C^0(\sphere^{n-1})}$ is bounded uniformly for $|\varepsilon|\le\varepsilon_0$.
 
 Finally, 
  $$\nabla^2F_\varepsilon
    =
    \bigl(2\psi(\varepsilon h)
    +4\varepsilon h\psi'(\varepsilon h)+\varepsilon^2h^2\psi''(\varepsilon h)\bigr)
    \nabla h\otimes\nabla h+
    \bigl(2h\psi(\varepsilon h)
    +\varepsilon h^2\psi'(\varepsilon h)\bigr) \nabla^2h,$$
 it also follows that $\|\nabla^2F_\varepsilon\|_{C^0(\sphere^{n-1})}$ is bounded uniformly for $|\varepsilon|\le\varepsilon_0$.
 
Hence, there exists $C>0$ independent of $\varepsilon$, such that
  $$
    \|F_\varepsilon\|_{C^2(\mathbb S^{n-1})}\leq C, \quad\text{for all }|\varepsilon|\le\varepsilon_0.
  $$
 Consequently,
  $$ \bigl\|(1+\varepsilon h)^\beta-1-\beta\varepsilon h\bigr\|_{C^2(\mathbb S^{n-1})}
    = |\varepsilon|^2\|F_\varepsilon\|_{C^2(\mathbb S^{n-1})} \leq C|\varepsilon|^2,\quad\text{for all }|\varepsilon|\le\varepsilon_0.$$
This proves the desired estimate.
\end{proof}

\subsection{The planar case}\label{planar}
\ 

For the sake of completeness, we include the proof of the planar case $\mu(2)=3$. The lower bound $\mu(2)\ge 3$ has appeared in, among other places, Bose~\cite{Bose1935ConvexOval}, Ehrhart~\cite{Ehrhart1955Minkowski}, and Myroshnychenko, Tatarko and Yaskin~\cite[Sec.~3]{MTY}; the upper bound $\mu(2)\le 3$ will be proved by an explicit example.
\begin{lemma}
  \label{lem:lower3}
   $\mu(2)\geq 3$.
\end{lemma}

\begin{proof}
  Let $K\subseteq \R^2$ be a convex body. 
  Without loss of generality, assume $c(K)=o$.
  Write $e_\alpha=(\cos\alpha,\sin\alpha)$ and
$$
g(\alpha)=\rho_K(e_\alpha)^3-\rho_K(e_{\alpha+\pi})^3,\quad \alpha\in \R.
$$
Since the radial function $\rho_K$ is continuous, so is $g$. Moreover, $g(\alpha)=0$ if and only if $c(K\cap e_{\alpha+\pi/2}^\perp)=o$.
Since $c(K)=o$, the identity $g(s+\pi)=-g(s)$ and the integration formula in polar coordinates give, for each $\alpha\in \R$,
\begin{equation}\label{eq:planar-2}
\int_{\alpha}^{\alpha+\pi}g(s)e_s\,ds=\int_0^{2\pi}\rho_K(e_s)^3e_s\,ds=3\int_K x\,dx =o.
\end{equation}
So,
$$
\int_{\alpha}^{\alpha+\pi}g(s)\langle e_s, e_{\alpha+\pi/2}\rangle\,ds=\int_{\alpha}^{\alpha+\pi}g(s)\sin(s-\alpha)\,ds=0.
$$
Since $\sin(s-\alpha)>0$ in $(\alpha,\alpha+\pi)$ and $g$ is continuous, it follows that for each $\alpha\in\R$, either $g$ vanishes identically on $[\alpha,\alpha+\pi]$, or $g$ takes both positive and negative values in $(\alpha,\alpha+\pi)$. If $g$ vanishes identically on  some $[\alpha,\alpha+\pi]$, then the identity $g(s+\pi)=-g(s)$ implies that $g$ vanishes identically on $\R$. So, either $g\equiv0$ on $\R$, or $g$ changes sign in $(\alpha,\alpha+\pi)$ for every $\alpha\in\R$.

If $g\equiv0$, then $c(K\cap e_{\alpha+\pi/2}^\perp)=o$ for every $\alpha$. Otherwise, fix $s\in\R$ such that $g(s)\neq 0$. 

Since $g(s+\pi)=-g(s)$,  there exists $\theta_1\in (s,s+\pi)$ with 
$g(\theta_1)=0$.
Since $g$ changes sign in $(\alpha,\alpha+\pi)$ for each $\alpha\in\R$,  there exists $\theta_2\in (\theta_1,\theta_1+\pi)$ with $g(\theta_2)=0$.
Assume $\theta_2$ is the only zero of $g$ in $(\theta_1,\theta_1+\pi)$. 
Without loss of generality, assume $g>0$ in $(\theta_1,\theta_2)$ and $g<0$ in $(\theta_2,\theta_1+\pi)$.
By $g(s+\pi)=-g(s)$, it follows that $g<0$ in $(\theta_1+\pi,\theta_2+\pi)$, and therefore $g\le 0$ in $(\theta_2,\theta_2+\pi)$,
contradicting that $g$ changes sign in $(\alpha,\alpha+\pi)$ for each $\alpha\in\R$.
So, there exists $\theta_3\in (\theta_1,\theta_1+\pi)\setminus\{\theta_2\}$ with $g(\theta_3)=0$. 
Hence, 
$$c(K\cap e_{\theta_1+\pi/2}^\perp)=c(K\cap e_{\theta_2+\pi/2}^\perp)=c(K\cap e_{\theta_3+\pi/2}^\perp)=o,$$
which implies that $\mu(2)\ge 3$.
\end{proof}

\begin{lemma}\label{lem:planar-upper}
   $\mu(2)\leq 3$.
\end{lemma}

\begin{proof}
  Let $S\subseteq \R^2$ be the equilateral triangle given by 
  $$
    S=\bigcap_{j=0}^2 \{x\in\R^2:\ip{x}{v_j}\leq 1\},
  $$
  where $v_j=\left(\cos\frac{2\pi j}{3},\sin\frac{2\pi j}{3}\right)$,
  $j=0,1,2$.
  Then its centroid  $c(S)=o$, and its radial function
  $$
    \rho_S(u)=\frac{1}{\max_j\ip{u}{v_j}},\quad \forall \  u\in \sphere^1. $$
Note that the chord $S\cap \operatorname{span}\{u\}$ of $S$ has its centroid at the origin $o$ if and only if
 $\rho_S(u)=\rho_S(-u)$, which holds if and only if
$$
  \max_j\ip{u}{v_j} = \max_j\ip{-u}{v_j} = - \min_j\ip{u}{v_j}.$$
Let $M=\max_j\ip{u}{v_j}$ and $m=\min_j\ip{u}{v_j}.$ Then
  $M=-m.$  
  
  If $M=0,$ then all the  $\ip{u}{v_j}=0,$ which is impossible for $u
	\neq 0$. Thus  $M>0,$ and $m<0.$
  Since $v_0+v_1+v_2=0$, it follows that 
  $ \ip{u}{v_0}+\ip{u}{v_1}+\ip{u}{v_2}=0$ for each $u\in \sphere^1,$ and therefore 
  $\ip{u}{v_j}=0$ for some $j \in \{0,1,2\}$. So, the direction  $u$ is parallel to the side of $S$ corresponding to $v_j$. Since $S$ has exactly three side directions, there are at most \emph{three} distinct  directions $u$ for which this can happen.  Therefore,  $\mu(2)\leq 3$.
\end{proof}

\vskip 15pt
\section{\bf Proof of main results}
\label{sec:main}
\vskip 5pt
To prove the main result, we first   prove that the hemispherical transform $\Hem$ is an \emph{isomorphism} on the space $C^\infty_{\mathrm{odd}}(\sphere^{n-1})$ of \emph{smooth odd functions} on $\sphere^{n-1}$, and   is a \emph{bounded} operator on $C^k(\sphere^{n-1})$ for each nonnegative integer $k$. Then, we prove that Morse functions are stable under $C^2$-perturbations and, in addition, that the number of critical points is stable as well. Finally, these results lead to the proof of Theorem~\ref{thm1}.

\vskip3pt
Write $C^\infty_{\mathrm{odd}}(\sphere^{n-1})$ for the space of smooth odd functions
$f:\sphere^{n-1}\to\R$, and write $C^\infty_{\mathrm{even}}(\sphere^{n-1})$ analogously.
Then $C^\infty(\sphere^{n-1})$ admits the direct sum decomposition
$$
  C^\infty(\sphere^{n-1})=C^\infty_{\mathrm{odd}}(\sphere^{n-1})\oplus C^\infty_{\mathrm{even}}(\sphere^{n-1}).$$

Recall that the \emph{hemispherical transform} of $f\in C(\sphere^{n-1})$ is 
$$ \Hem f(u)=
  \int_{\sphere^{n-1}}
  \mathbbm 1_{\{ \ip{\theta}{u}\geq 0\}} f(\theta)\dd\Haus^{n-1}(\theta),\quad  \forall \ u\in \sphere^{n-1}.$$
  The action of $\Hem$ on even functions is very simple. Indeed,
  if $f$ is even, the change of variables 
		$\theta\mapsto -\theta$, together with $\dd\Haus^{n-1}(-\theta)=\dd\Haus^{n-1}(\theta)$, gives
$$ \Hem f(u)=\frac{1}{2}\int_{\sphere^{n-1}}f(\theta)\dd\Haus^{n-1}(\theta),\quad  \forall \ u\in \sphere^{n-1},
$$
which is constant. In what follows, we concentrate on the action of $\Hem$ on $C^\infty_{\mathrm{odd}}(\sphere^{n-1})$.

\begin{lemma}\label{lem14}
 $\Hem:C^\infty_{\mathrm{odd}}(\sphere^{n-1})
    \longrightarrow
    C^\infty_{\mathrm{odd}}(\sphere^{n-1})$
  is a linear isomorphism for all $n\geq 3$.
  Moreover, $f\in C^\infty_{\mathrm{odd}}(\sphere^{n-1})$ has no degree-one spherical harmonic component if and only if  $\Hem f$ has none.
\end{lemma}

\begin{proof}
In \cite[(2.10) and (2.14)]{Rubin}, Rubin showed that for each odd integer $m \geq 1$
and each  spherical harmonic $Y_{m,k}$ of degree $m,$
  $\Hem Y_{m,k}=\lambda_m Y_{m,k},$
  with eigenvalues
  $$\lambda_m=\pi^{(n-2)/2}(-1)^{(m-1)/2}
    \frac{\Gamma(m/2)}{\Gamma((m+n)/2)}.
  $$
 Since $\lambda_m\neq 0$ for all odd $m$,   $\Hem$ is \emph{injective} on $C^\infty_{\mathrm{odd}}(\sphere^{n-1}).$  
  Indeed, assume  $f \in C^\infty_{\mathrm{odd}}(\sphere^{n-1})$ and  $$f=\sum_{\substack{ m\ \mathrm{odd}}}\sum_{k=1}^{N(n,m)} c_{m,k} Y_{m,k}.$$  If $\Hem f=0,$ then 
 $$\sum_{\substack{ m\ \mathrm{odd}}}\sum_{k=1}^{N(n,m)} c_{m,k} \lambda_mY_{m,k}=0,$$
 which forces $c_{m,k}=0$ for all $m,k$ by the linear independence of $Y_{m,k}.$ Hence, $f=0.$
 
 Next,  we show that  $\Hem$ is \emph{surjective} on $C^\infty_{\mathrm{odd}}(\sphere^{n-1}).$ 
   For any $g\in C^\infty_{\mathrm{odd}}(\sphere^{n-1})$, write its spherical harmonic expansion (only odd degrees contribute) as 
  $$ g=\sum_{\substack{ m\ \mathrm{odd}}}\sum_{k=1}^{N(n,m)} g_{m,k} Y_{m,k}.
  $$
  Define
  $$f:=
    \sum_{\substack{ m\ \mathrm{odd}}}\sum_{k=1}^{N(n,m)}
    \lambda_m^{-1} g_{m,k} Y_{m,k},$$
   where $\lambda_m$ is as above and nonzero for all odd $m$. Since only odd-degree spherical harmonics appear in its expansion, 
$f$ is odd. Next, We aim to show that $f$ is smooth. If so, then $f$ lies in $ C^\infty_{\mathrm{odd}}(\sphere^{n-1})$ and   immediately gives that  $\Hem f = g.$

Recall that on $\sphere^{n-1}$ the Sobolev norm $H^s$ for $s \geq 0$ can be characterized via spherical harmonic coefficients (see Atkinson and Han~\cite[(3.98)]{AtkinsonHan2012}) by
 $$\|g\|_{H^s}^2 =
    \sum_{m,k} \big(m+\frac{n-2}{2}\big)^{2s}|g_{m,k}|^2.
  $$
 Since $g$ is smooth, it lies in the Sobolev space $H^s(\sphere^{n-1})$. So the  sum is finite for every $s \geq 0.$

  Applying the Stirling formula to the eigenvalues $\lambda_m$, we obtain the asymptotic estimate
$$|\lambda_m|^{-2}=O(m^n), \quad \text{as}\ m \rightarrow \infty,\ m \ \text{odd},$$
  and hence 
$$|\lambda_m|^{-2}=O((m+\frac{n-2}{2})^{n}), \quad \text{as}\ m \rightarrow \infty,\ m \ \text{odd}.$$
 So for every $s \geq 0$,
 \begin{align*}
    \|f\|_{H^s}^2
    =&
    \sum_{m,k} \big(m+\frac{n-2}{2}\big)^{2s}|\lambda_m|^{-2}|g_{m,k}|^2 \\
    =& \,O\big(\sum_{m,k} \big(m+\frac{n-2}{2}\big)^{2s+n}|g_{m,k}|^2 \big)\\
    = &\, O\big(\|g\|_{H^{s+n/2}(\sphere^{n-1})}^2\big)<+\infty,
  \end{align*}
 which implies that  $f \in  H^s(\sphere^{n-1})$ for every $s \geq 0$. By the Sobolev embedding theorem on compact manifolds, we obtain
  $$\bigcap_{s\ge0}H^s(\sphere^{n-1}) = C^\infty(\sphere^{n-1}).$$
  Therefore,  $f\in C^\infty_{\mathrm{odd}}(\sphere^{n-1})$. 

 Since $\Hem Y_{m,k}=\lambda_m Y_{m,k}$, it follows that $\Hem$ is diagonalized by spherical harmonics.
 In particular, $\Hem$ acts on $\Harm^1$ as multiplication by   $\lambda_1 \neq 0$. So, $\pi_1 (\Hem f)=\Hem (\pi_1 f)=\lambda_1\pi_1 (f),$ which implies that
  $\Hem f$ has no degree-one component if and only if $f$ has no degree-one component.
\end{proof}

To prove  the  transform $\Hem$ is a bounded operator on  $C^k(\sphere^{n-1})$ for every $k$,  we prove the   boundedness of convolution on compact Lie groups, and lift $\Hem$ to a convolution operator.

\begin{lemma}\label{lem9}
Let $G$ be a compact Lie group and $\nu$ be a finite positive Borel measure on $G$. For each integer $k\geq 0$, define the operator $\mathcal{T}_\nu$ on $C^k(G)$ by
  $$(\mathcal{T}_\nu f)(\xi)=\int_G f(\xi\eta)\dd\nu(\eta), \quad \text{for}\  f\in C^k(G), \ \xi\in G.$$
  Then $\mathcal{T}_\nu:C^k(G)\to C^k(G)$ is bounded.
\end{lemma}

\begin{proof}
  Choose a basis $X_1,\dots,X_N$ of the Lie algebra of $G$, and let $X_1^R,\dots,X_N^R$
  be the corresponding right-invariant vector fields.  Since $f \in C^k(G)$ and the right-invariant vector fields act as differential operators of order $|I|$, the map $(\xi,\eta)\mapsto X_I^Rf(\xi\eta)$
  is continuous on the compact space $G\times G$ 
  for every multi-index $I$ with $|I|\leq k$. Since $\nu$ is a finite measure, the Lebesgue dominated convergence theorem  gives
  that for all $\xi \in G,$
  $$ X_I^R(\mathcal{T}_\nu f)(\xi)=\int_G X_I^R f(\xi\eta)\dd\nu(\eta).$$
 Since  $G$  is compact and the integral is continuous in $\xi,$ it follows that $X_I^R(\mathcal{T}_\nu f) \in C^0(G)$ for all $|I|\leq k$, which proves that  $\mathcal{T}_\nu f \in C^k(G).$
 
 Since the right translation $\xi \mapsto \xi\eta$ is a diffeomorphism of $G$ onto itself, it follows that 
 $$
	\sup_{\xi \in G} |X_I^Rf(\xi\eta)|=\sup_{g \in G} |X_I^Rf(g)|=\|X_I^R f\|_{C^0(G)}.
$$
 Consequently,
  $$\|X_I^R(\mathcal{T}_\nu f)\|_{C^0(G)} \leq
    \nu(G)\|X_I^R f\|_{C^0(G)}.$$
  
 Finally, since $\{X_1^R,\dots,X_N^R\}$
 forms a global frame of the tangent bundle, the standard $C^k$ norm on $G$ is equivalent to $
	{\sum }_{|I|\leq k} \|X_I^R(\cdot) \|_{C^0(G)}.$
Summing  the  estimates over all  $|I|\leq k$ yields
  $$
    \|\mathcal{T}_\nu f\|_{C^k(G)}\leq \nu(G)\|f\|_{C^k(G)}.
  $$
  Thus, $\mathcal{T}_\nu:C^k(G)\to C^k(G)$ is bounded.
\end{proof}

\begin{lemma}\label{lem1}
  For every integer $k\geq 0$, there exists a constant $C=C(n,k)>0$ such that
  $$
    \|\Hem f\|_{C^k(\sphere^{n-1})}
    \leq
    C\|f\|_{C^k(\sphere^{n-1})}, \quad \forall \ 
  f\in C^k(\sphere^{n-1}).
  $$
\end{lemma}

\begin{proof}
Let $G=\operatorname{SO}(n)$ and let $\pi:G\to \sphere^{n-1}$ be given by $\pi(\xi)=\xi e_n$, where $e_n=(0,0,\ldots, 1)$. For 
$f\in C^k(\sphere^{n-1})$, define its pullback $\widetilde f=f\circ \pi \in C^k(G)$. Let
$
E=\{\theta\in\sphere^{n-1}:\ip{\theta}{e_n}\geq 0\}.
$

First, we construct a finite positive Borel measure $\nu$ on $G$
 such that its pushforward satisfies
$\pi_*\nu=\Haus^{n-1}|_E$. Indeed, 
let $m_G$ be the  Haar probability measure on $G$.
Since $\pi(\xi)=\xi e_n$ identifies $\sphere^{n-1}$ with
$G/\mathrm{SO}(n-1)$, the pushforward $\pi_*m_G$ is precisely the normalized
rotation-invariant measure on $\sphere^{n-1}$, that is,
$$
  \pi_*m_G=\frac{1}{\Haus^{n-1}(\sphere^{n-1})}\Haus^{n-1}.
$$
For any Borel set $B\subset G,$ define 
$$\nu(B) =
  \Haus^{n-1}(\sphere^{n-1})\,
  m_G\bigl(B\cap \pi^{-1}(E)\bigr).
$$
Then $\nu(G)=\frac{1}{2}\Haus^{n-1}(\sphere^{n-1})<\infty,$  and  for any Borel set $A\subset\sphere^{n-1}$,
$$
  \pi_*\nu(A)=
  \Haus^{n-1}(\sphere^{n-1})m_G(\pi^{-1}(A\cap E))
  =\Haus^{n-1}(A\cap E).$$
That is, $\pi_*\nu=\Haus^{n-1}|_E$. 
By the $\mathrm{SO}(n)$-invariance of the spherical measure,  we have

\begin{align*}
\widetilde{\Hem f}(\xi)
&=
\Hem f(\xi e_n)
=
\int_{\xi E} f(\theta)\dd\Haus^{n-1}(\theta)\\
& =
\int_E f(\xi\theta)\dd\Haus^{n-1}(\theta)=
\int_G f(\xi\eta e_n)\dd\nu(\eta) \\
& =
\int_G \widetilde f(\xi \eta)\dd\nu(\eta)
=
\mathcal{T}_\nu \widetilde f(\xi), \quad \ \forall \ \xi\in G.
\end{align*}
Since $\pi:G\to\sphere^{n-1}$ is a smooth submersion with compact fibre $\mathrm{SO}(n-1)$, there exists $\delta>0$ such that for every $f\in C^k(\sphere^{n-1})$,
$$
\delta\|\widetilde f\|_{C^k(G)}
\leq
\|f\|_{C^k(\sphere^{n-1})}
\leq
\delta^{-1}\|\widetilde f\|_{C^k(G)}.
$$
By Lemma~\ref{lem9}, the operator $\mathcal{T}_\nu$ is bounded on $C^k(G)$. Thus,
\begin{align*}
\|\Hem f\|_{C^k(\sphere^{n-1})}
&\leq
\delta^{-1}\|\widetilde{\Hem f}\|_{C^k(G)} =
\delta^{-1}\|\mathcal{T}_\nu \widetilde f\|_{C^k(G)} 
\leq
\delta^{-1}C^{\prime}\|\widetilde f\|_{C^k(G)}
\leq
\delta^{-2}C^{\prime}\|f\|_{C^k(\sphere^{n-1})}.
\end{align*}
Let $C=\delta^{-2}C^{\prime}.$ This completes the proof.
\end{proof}

\vskip3pt

Let $M$ be a smooth manifold and $f\in C^2(M)$. If $df(p)=0$, then $p$ is called a \emph{critical point} of $f$. Let $(x^i)$ be local coordinates around $p$. The \emph{Hessian} of $f$ at $p$, denoted by $\operatorname{Hess}_M f(p)$, is the symmetric bilinear form on $T_pM$ whose matrix with respect to the coordinate basis $(\frac{\partial}{\partial x^i}|_p)_i$ is
$$
(\frac{\partial^2 f}{\partial x^i\partial x^j}(p))_{i,j}.
$$
A critical point $p$ is called \emph{nondegenerate} if this matrix is invertible. Please consult  Milnor~\cite{Milnor} for more background on Morse function. 

\begin{definition}
  Let $M$ be a smooth manifold and $f\in C^2(M)$. The function $f$ is called a \emph{Morse function} if every critical point of $f$ is nondegenerate.
\end{definition}

The property of being Morse is invariant under diffeomorphisms. More precisely, given a diffeomorphism $\Phi:N\rightarrow M$, a function $f:M\rightarrow \R$ is Morse if and only if $f\circ \Phi:N\rightarrow \R$ is Morse. Moreover, if $f$ has critical points $\{p_i\}$, then $f\circ \Phi$ has critical points $\{\Phi^{-1}(p_i)\}$.

Our proof of Theorem~\ref{thm1} rely  on the stability of the number of critical points of Morse functions under sufficiently small perturbations in the $C^2$ topology. 

\begin{lemma}\label{lem16}
 Suppose $U\subseteq \R^n$ is open and $f\in C^2(U)$. If $f$ has exactly one nondegenerate critical point $p\in U$, then there exist a real number $\delta>0$ and an open ball $V$ with $p\in V$ and $\bar V\subset U$, 
  such that for every $g\in C^2(U)$ satisfying  $\|g-f\|_{C^2(\bar V)}<\delta$, the function $g$ has exactly one nondegenerate critical point in $V$. 
\end{lemma}

\begin{proof}
Without loss of generality, assume that $p=0$. Since $\nabla^2 f(0)$ is nonsingular, by continuity there exists an open ball
$V$ centered at $0$ with $ \bar V\subset U$ such that $\nabla^2 f$ is nonsingular on $\bar V$. Shrinking 
$V$ if necessary, we may also assume that $\nabla f \neq 0$ on $\partial V.$
 
\textbf{Step 1.} Show that there is a point $x\in V$ such that $x$ is a nondegenerate critical point of $g$.

By the continuity of $\nabla^2 f$ and compactness of $\bar V$, 
we choose $\delta>0$ small enough that for every $g\in C^2(U)$ with $\|g-f\|_{C^2(\bar V)}<\delta$, the following claims (i) and (ii) hold for all $t\in [0,1]$.
\begin{itemize}
    \item[] (i). $\nabla f+t\nabla(g-f)\neq 0$ on $\partial V$;
\item[] (ii). $\nabla^2 f+t\nabla^2(g-f)$ is nondegenerate on $\bar V$.
\end{itemize}

Let
$$
F(t,x)
=
\nabla f(x)+t\nabla (g-f)(x),\quad \text{for}\ (t,x)\in [0,1]\times V.
$$
We aim to show that $F(1,x)=0$ for some $x\in V$.
Let
$$
S=\{t\in[0,1]: \text{ there exists } x\in V
\text{ such that } F(t,x)=0\}.
$$
Clearly $0\in S$ since $F(0,0)=0$. In the following, we prove that $S$ is open and closed.

Assume $t_0\in S$, and choose 
$x_0\in V$ such that $F(t_0,x_0)=0$. Then
$$
\nabla^2 f(x_0)+t_0\nabla^2 (g-f)(x_0)
$$
is nondegenerate by the choice of $\delta$. Since $x_0\in V$ and $V$ is open, by the implicit
function theorem, there exists a neighborhood $\mathcal{N}(t_0)$ of $t_0$ in $[0,1]$ such that for every $t\in\mathcal{N}(t_0)$, there exists an $x\in V$ that solves $F(t,x)=0$. So, $\mathcal{N}(t_0)\subseteq S$, and $S$ is open in $[0,1]$.

Assume $t_j\in S$ and $t_j\to t$. Choose $x_j\in V$ such that $F(t_j,x_j)=0$. By compactness, after passing to a subsequence, $x_j\to x\in\bar V$. By continuity of $F$, we have $F(t,x)=0$. If $x\in\partial V$, this contradicts (i). Thus, $x\in V$, and $S$ is closed in $[0,1]$.

Hence, $S$ is nonempty, open and closed, and therefore $S=[0,1]$.
In particular, $1\in S$, as desired. Moreover, by (ii), the Hessian of $g$
is nondegenerate at any critical point of $g$ in $V$.

\textbf{Step 2.} Show that there is exactly one $x\in V$ with $\nabla g(x)=0$ for sufficiently small   $\delta$. 

 For 
 $\forall \  m\in \mathbb{N}$, assume there exists  $g_m \in  C^2(U)$ with $\|g_m-f\|_{C^2(\bar V)}\leq \frac1m$ and  two distinct  points $x_m, y_m \in V$ such that $\nabla g_m(x_m)=\nabla g_m(y_m)=0$. 
Then $\lim_{m\to \infty}x_m=0$. 
Otherwise, there exists a subsequence of $\{x_m\}$ converging to $x\in \bar V\setminus\{0\}$ with $\nabla f(x)=0$, 
contradicting the assumption. Similarly, $\lim_{m\to \infty}y_m=0$.
Let
$$ u_m:=\frac{x_m-y_m}{\|x_m-y_m\|}, \quad m\in \mathbb{N}.$$

By the compactness of $\bar V$, $\nabla^2 f$ is uniformly continuous on $\bar V$. 
Since $x_m,y_m\to 0$, we have
\begin{equation}\label{eq18}
  \sup_{0\leq t\leq 1}
  \left\|\nabla^2 f(y_m+t(x_m-y_m))-\nabla^2 f(0)\right\|\leq \omega_{\nabla^2 f}(\max(\|x_m\|,\|y_m\|))\to 0, \quad \text{as}\ m\rightarrow \infty,
\end{equation}
where $\omega_{\nabla^2 f}$ is the modulus of continuity of $\nabla^2 f$ on $\bar V$.
Thus,
\begin{align*}
\frac{\nabla g_m(x_m)- \nabla g_m(y_m)}{\|x_m-y_m\|}
&= \int_0^1 \nabla^2 g_m(y_m+t(x_m-y_m)) u_m \dd t\\
&= \int_0^1 \nabla^2 f(y_m+t(x_m-y_m)) u_m \dd t+O\big(\frac{1}{m}\big)\\
&= \nabla^2 f(0)u_m+o(1),
\end{align*}
where we used $\|g_m-f\|_{C^2(\bar V)}\leq \frac1m$ for the second equality and \eqref{eq18} for the last equality.

However,
$$\|\nabla^2 f(0)u_m\| \geq \|(\nabla^2 f(0))^{-1}\|^{-1},$$
where $\|\cdot\|$ is the operator norm.
Hence
$$
0=\liminf_{m\to \infty}\left\|\frac{\nabla g_m(x_m)- \nabla g_m(y_m)}{\|x_m-y_m\|}\right\|
\geq \|(\nabla^2 f(0))^{-1}\|^{-1}>0,
$$
a contradiction.
\end{proof}

\begin{lemma}\label{lem5}
  Let $M$ be a compact smooth manifold without boundary and let $f\in C^2(M)$ be Morse. Then there exists $\delta>0$ such that every $g\in C^2(M)$ with $\|g-f\|_{C^2(M)}<\delta$ is also Morse and has exactly the same number of critical points as $f$.
\end{lemma}

\begin{proof}
The critical set of $f$ is closed and, since $f$ is Morse, discrete. It is therefore a compact discrete set and hence finite. The assumption $\partial M=\varnothing$ first ensures that the maximum and minimum points of $f$, which exist by compactness, are interior points and hence critical points. Thus, the set of critical points is nonempty; denote it by $\{p_1,\dots,p_\ell\}$, where $\ell\geq 1$.

The assumption $\partial M=\varnothing$ is also used here: it guarantees that every $p_j$ has a coordinate neighborhood modeled on an open subset of $\R^{\dim M}$, rather than on a half-space, so that the local stability result in Lemma~\ref{lem16} applies. Choose pairwise disjoint coordinate neighborhoods $Q_j$ of $p_j$ such that $p_j$ is the only critical point of $f$ in $Q_j$. Fix coordinate maps $\phi_j:Q_j\to\phi_j(Q_j)\subset\R^{\dim M}$ with $\phi_j(p_j)=0$, and shrink $Q_j$ so that the Hessian of $f\circ\phi_j^{-1}$ is nondegenerate throughout $\phi_j(Q_j)$. For each $j$, Lemma~\ref{lem16}, applied to $f\circ\phi_j^{-1}$, gives an open ball $V_j$ with $\bar V_j \subset\phi_j(Q_j)$ and a number $\eta_j>0$ such that every $q\in C^2(\phi_j(Q_j))$ satisfying
$$
\|q-f\circ\phi_j^{-1}\|_{C^2(\bar V_j)}<\eta_j
$$
has exactly one nondegenerate critical point in $V_j$. Since there are only finitely many charts and restriction and pullback by each fixed coordinate map are continuous in the $C^2$ topology, there exists $\delta_0>0$ such that $\|g-f\|_{C^2(M)}<\delta_0$ implies all of these local inequalities with $q=g\circ\phi_j^{-1}$. Consequently, $g$ has exactly one nondegenerate critical point in each
  $$
    U_j:=\phi_j^{-1}(V_j).
  $$

Let $F=M\setminus\bigcup_{j=1}^{\ell}U_j$. If $F=\varnothing$, the conclusion follows by taking $\delta=\delta_0$. Otherwise, fix a Riemannian metric on $M$. Since $F$ is compact and $df$ does not vanish on $F$, we have
$$
m:=\min_{x\in F}\|df(x)\|>0.
$$
Choose $\delta_1>0$ such that $\|g-f\|_{C^2(M)}<\delta_1$ implies $\|d(g-f)\|_{C^0(M)}<m/2$. Then $\|dg(x)\|\geq m/2$ for every $x\in F$, so $g$ has no critical point in $F$. Taking $\delta=\min\{\delta_0,\delta_1\}$, we conclude that $g$ has exactly $\ell$ critical points, all of which are nondegenerate.
\end{proof}

Indeed, the fact that a sufficiently small $C^2$ perturbation of a Morse function is again Morse is standard.  However, we need the above \emph{stronger} conclusion that the perturbation has exactly the \emph{same} number of critical points. 

One of the central ingredients to  prove  Theorem~\ref{thm1}, is the construction of smooth odd Morse function on $\sphere^{n-1}$ with exactly two critical points and vanishing degree-one component. 

\begin{lemma}\label{lem:alpha}
Let $n\geq 3$. There exists a smooth even function $\alpha:[-1,1]\to\R$ such that
\begin{equation}\label{eq:alpha-moment}
\int_{-1}^1 (1-t^2)^{(n-1)/2}e^{i\alpha(t)}\dd t=0.
\end{equation}
\end{lemma}

\begin{proof}
Let $w(t)=(1-t^2)^{(n-1)/2}$, $W=\int_0^1w(t)\dd t$, and
$$
s(t)=W^{-1}\int_0^t w(\tau)\dd\tau,
\quad 0\leq t\leq 1.
$$
Then $s$ is smooth in $(0,1)$ and increasing on $[0,1]$, with $s(0)=0$ and $s(1)=1$.

 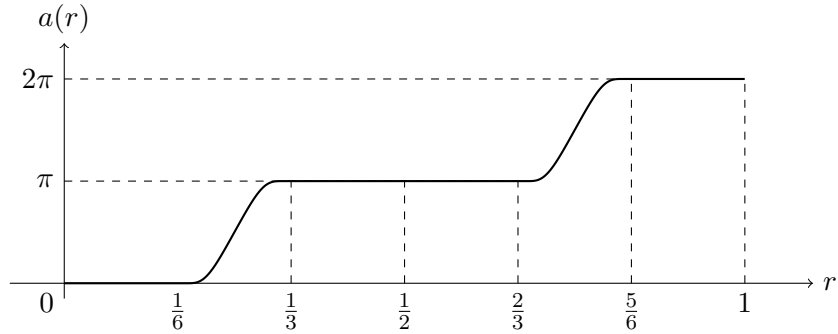
\begin{figure}[H]
\centering

\pgfmathdeclarefunction{smoothstep}{1}{%
  \pgfmathparse{#1}%
  \ifdim\pgfmathresult pt<0.000001pt
    \pgfmathparse{0}%
  \else
    \ifdim\pgfmathresult pt>0.999999pt
      \pgfmathparse{1}%
    \else
      \pgfmathparse{%
        exp(-1/(#1))/
        (exp(-1/(#1))+exp(-1/(1-(#1))))
      }%
    \fi
  \fi
}

\begin{tikzpicture}[x=9cm,y=1.35cm]

  \draw[->] (-0.08,0) -- (1.1,0) node[right] {$r$};
  \draw[->] (0,-0.15) -- (0,2.35) node[above] {$a(r)$};

  \draw[
    thick,
    samples=401,
    domain=0:1,
    variable=\r
  ]
  plot
  (
    {\r},
    {smoothstep(6*\r-1)+smoothstep(6*\r-4)}
  );

  \node[below left] at (0,0) {$0$};
  \node[below] at (1/6,0) {$\frac16$};
  \node[below] at (1/3,0) {$\frac13$};
  \node[below] at (1/2,0) {$\frac12$};
  \node[below] at (2/3,0) {$\frac23$};
  \node[below] at (5/6,0) {$\frac56$};
  \node[below] at (1,0) {$1$};

  \node[left] at (0,1) {$\pi$};
  \node[left] at (0,2) {$2\pi$};

  \draw[dashed,thin] (1/3,0) -- (1/3,1);
  \draw[dashed,thin] (1/2,0) -- (1/2,1);
  \draw[dashed,thin] (2/3,0) -- (2/3,1);
  \draw[dashed,thin] (5/6,0) -- (5/6,2);
  \draw[dashed,thin] (1,0)   -- (1,2);
  \draw[dashed,thin] (0,1) -- (2/3,1);
  \draw[dashed,thin] (0,2) -- (1,2);

\end{tikzpicture}

\caption{Graph of  the function $a$.}
\label{Fig1}
\end{figure}

Let $a\in C^\infty([0,1])$ be as in Figure~\ref{Fig1}.
In particular, we require that $a=0$ near $0$; $a=\pi$ near $\frac12$;  $a=2\pi$ near $1$; and
\begin{equation}\label{eq21}
  a(x+\frac12)=\pi+a(x), \  \forall \ x\in[0,\frac{1}{2}].
 \end{equation}
 
Define $\alpha(t)=a(s(t))$ for $0\leq t\leq1$, and extend it evenly to
$[-1,1]$.
Since $a$ is constant near $0$ and $1$, the extension is globally smooth.
Using $w(t)\dd t=W\dd s(t)$ and \eqref{eq21}, we have
$$
\begin{aligned}
\int_{-1}^1 w(t)e^{i\alpha(t)}\dd t
=2W\int_0^1 e^{ia(r)}\dd r 
=2W\int_0^{\frac{1}{2}} e^{ia(r)}\dd r
  +2W\int_0^{\frac{1}{2}} e^{i(\pi+a(r))}\dd r=0.
\end{aligned}
$$
This proves the lemma.
\end{proof}

Now we construct the desired Morse function $f$. 
We start with an \emph{odd} linear function that is Morse with exactly two nondegenerate critical points but has a nonzero degree-one component. We then compose it with a suitably \emph{odd} diffeomorphism of the sphere, which preserves the Morse property and the number of critical points while arranging for the first moment to vanish.

Fix $n\geq 3$. Write each point $u$ of $\sphere^{n-1}\subset\R^n$ as
$$
u=(x_1,x_2,y,t)
\in \R\oplus\R\oplus\R^{n-3}\oplus\R,
\qquad
x_1^2+x_2^2+|y|^2+t^2=1.
$$
We regard $t$ as the altitude coordinate of $u$. The construction rotates the $(x_1,x_2)$-plane by an angle $\alpha(t)$ determined by Lemma \ref{lem:alpha}.

\begin{lemma}\label{prop2}
Let $n\geq 3$. There exists a Morse function
$f\in C^\infty_{\mathrm{odd}}(\sphere^{n-1})$ which has exactly two critical points $p$ and $-p$, and satisfies
$\int_{\sphere^{n-1}} u f(u)\dd\Haus^{n-1}(u)=0.$  
\end{lemma}

\begin{proof}
We divide the proof into two steps.

\textbf{Step 1.} Construct a Morse function $f$.

Let $\alpha$ be as in Lemma~\ref{lem:alpha}. For $\theta\in\R$, write
$$
\phi_\theta=
\begin{pmatrix}
\cos\theta&-\sin\theta\\
\sin\theta&\cos\theta
\end{pmatrix}.
$$
Define
$$
\Phi:\sphere^{n-1}\to\sphere^{n-1},
\quad
\Phi(x_1,x_2,y,t)=(\phi_{\alpha(t)}(x_1,x_2),y,t).
$$
Then $\Phi$ is a smooth diffeomorphism, with inverse obtained by replacing
$\alpha(t)$ by $-\alpha(t)$. Since $\alpha$ is even, $\Phi$ is odd.
Let $f$ be the first coordinate of $\Phi$, i.e.,
$$
f(x_1,x_2,y,t)=x_1\cos\alpha(t)-x_2\sin\alpha(t), \quad (x_1,x_2,y,t)\in \sphere^{n-1}.
$$

\textbf{Step 2.} Show that $f$ is the desired function.

  By the construction, $f=x_1\circ\Phi$.
  Since the function $x_1$ is Morse with exactly two critical points ($\pm e_1=(\pm 1,0,\ldots,0)$) and $\Phi$ is a diffeomorphism,
$f=x_1\circ\Phi$ is Morse with precisely the two critical points
$\Phi^{-1}(e_1)$ and $\Phi^{-1}(-e_1)$. Also, $f$ is odd and smooth because both $x_1$ and $\Phi$ are odd and smooth. It remains to prove that the first moment of $f$ vanishes.
For
$$
u=(x_1,x_2,y,t)\in \sphere^{n-1}\setminus\{\pm e_n\},
$$
write
$$
v=(v_1,\dots,v_{n-1})
=
\frac{(x_1,x_2,y)}{\sqrt{1-t^2}}
\in\sphere^{n-2}.
$$
Then $u_i=\sqrt{1-t^2}\,v_i$ for $1\leq i\leq n-1$, $u_n=t$, and
$$
f(u)=\sqrt{1-t^2}\,
\bigl(v_1\cos\alpha(t)-v_2\sin\alpha(t)\bigr).
$$

So
\begin{align*}
    &\int_{\sphere^{n-1}} u_nf(u)\dd\Haus^{n-1}(u)
    =&
    \int_{\sphere^{n-2}}\int_{-1}^1
    t\bigl(v_1\cos\alpha(t)-v_2\sin\alpha(t)\bigr)
    (1-t^2)^{\frac{n-2}{2}}
    \dd t\dd\Haus^{n-2}(v)
    =0,
\end{align*}
where the last equality holds because $\alpha$ is even;
for $1\leq i\leq n-1$,
\begin{align*}
    &\int_{\sphere^{n-1}} u_if(u)\dd\Haus^{n-1}(u)\\
    =&
    \int_{\sphere^{n-2}}\int_{-1}^1
    v_i\bigl(v_1\cos\alpha(t)-v_2\sin\alpha(t)\bigr)
    (1-t^2)^{(n-1)/2}
    \dd t\dd\Haus^{n-2}(v) \\
    =&
    2\int_{\sphere^{n-2}} v_i
    \int_0^1
    \bigl(v_1\cos\alpha(t)-v_2\sin\alpha(t)\bigr)
    (1-t^2)^{(n-1)/2}
    \dd t\dd\Haus^{n-2}(v)=0,
\end{align*}
where the last equality follows from the real and imaginary parts of
\eqref{eq:alpha-moment}.
Therefore, 
$$
\int_{\sphere^{n-1}}uf(u)\dd\Haus^{n-1}(u)=0
$$
and the lemma follows.
\end{proof}

\begin{proof}[Proof of Theorem~\ref{thm1}]
 Lemmas~\ref{lem:lower3} and~\ref{lem:planar-upper} give $\mu(2)=3$, and Lemma~\ref{lem:lower1} gives $\mu(n)\geq 1$ for all $n\geq 3$. It therefore suffices to prove $\mu(n)\leq 1$ for $n\geq 3$.
 Fix $n\geq 3$. We will construct a family of smooth strictly convex bodies in $\R^n$ with exactly one hyperplane section whose centroid coincides with that of the entire body, proving $\mu(n)\leq 1$. 

\textbf{Step 1.} We present the construction of the desired family of convex bodies in this first step.
By Lemma~\ref{prop2}, there exists a function
$f\in C^\infty_{\mathrm{odd}}(\sphere^{n-1})$ that is Morse, has exactly two critical points, and satisfies
$$
\int_{\sphere^{n-1}}u f(u)\dd\Haus^{n-1}(u)=0.
$$
By Lemma~\ref{lem14}, the hemispherical transform is an isomorphism on smooth odd functions. Hence
$h:=\Hem^{-1}f$
is a smooth odd function on $\sphere^{n-1}$. For sufficiently small $\eps>0$, we have $1+\eps h>0$. 

Let
\begin{equation}\label{eq:Keps-main}
K_\eps
=
\left\{ru:u\in\sphere^{n-1},\ 0\leq r\leq
(1+\eps h(u))^{1/(n+1)}\right\}.
\end{equation}
We will show that for all sufficiently small $\varepsilon>0$, $K_\varepsilon$ has exactly one hyperplane section whose centroid coincides with the centroid of $K_\varepsilon$ itself.

\textbf{Step 2.} We claim that $K_\varepsilon$ is smooth and strictly convex for all sufficiently small $\varepsilon>0$. Indeed,
Lemma~\ref{lem4}, together with $\beta=1/(n+1)$, gives
$$
(1+\eps h)^{1/(n+1)}
=1+\frac{\eps}{n+1}h+O(\eps^2)
\quad\text{in }C^2(\sphere^{n-1}).
$$
Consequently, the radial function
$\rho_{K_\eps}(u)=(1+\eps h(u))^{1/(n+1)}$
converges to  $1$ (i.e., the radial function of  the Euclidean unit ball)  in the $C^2$ topology. By Lemma~\ref{lem3},  it follows that $K_\eps$ is a smooth convex body with positive Gauss curvature everywhere, and hence is strictly convex.

\textbf{Step 3.} We show that $c(K_\varepsilon)=0$.
The vanishing first moment of $f$ is equivalent, by \eqref{eq:no-H1}, to the absence of a degree-one spherical harmonic component. Since $\Hem h=f$, the last assertion of Lemma~\ref{lem14} implies that $h$ also has no degree-one component. Therefore,
\begin{equation}\label{eq6}
\int_{\sphere^{n-1}}u h(u)\dd\Haus^{n-1}(u)=0.
\end{equation}
Using polar coordinates and \eqref{eq:Keps-main}, we obtain
$$
\begin{aligned}
\int_{K_\eps}x\dd\Haus^n(x)
&=\int_{\sphere^{n-1}}
\big(\int_0^{\rho_{K_\eps}(u)}u\,t^n\dd t\big)
\dd\Haus^{n-1}(u)\\
&=\frac1{n+1}\int_{\sphere^{n-1}}
u\rho_{K_\eps}(u)^{n+1}\dd\Haus^{n-1}(u)\\
&=\frac1{n+1}\int_{\sphere^{n-1}}
u(1+\eps h(u))\dd\Haus^{n-1}(u)=0,
\end{aligned}
$$
where the last equality follows from \eqref{eq6} and the symmetry of $\sphere^{n-1}$. Thus, $c(K_\eps)=0$.

\textbf{Step 4.} We prove that $K_\varepsilon$ has exactly one hyperplane section whose centroid coincides with the centroid of $K_\varepsilon$ itself.
By formula~\eqref{eq3} and Lemma~\ref{lem4}, now applied with $\beta=n/(n+1)$, the half-space volume function of $K_\eps$ satisfies
\begin{align}
A_{K_\eps}
&=\frac1n\Hem\big((1+\eps h)^{n/(n+1)}\big)\nonumber\\
&=\frac1n\Hem\big(1+\frac{n}{n+1}\eps h+R_\eps\big)\nonumber\\
&=\frac{\Haus^{n-1}(\sphere^{n-1})}{2n}
+\frac{\eps}{n+1}f+\frac1n\Hem(R_\eps),
\label{eq20}
\end{align}
where $\|R_\eps\|_{C^2(\sphere^{n-1})}=O(\eps^2)$. Define
$$
g_\eps
=\frac{n+1}{\eps}
\big(A_{K_\eps}-\frac{\Haus^{n-1}(\sphere^{n-1})}{2n}\big).
$$
The $C^2$ boundedness of the hemispherical transform in Lemma~\ref{lem1}, together with \eqref{eq20}, yields
\begin{equation}\label{eq7}
\|g_\eps-f\|_{C^2(\sphere^{n-1})}
=\frac{n+1}{n\eps}\|\Hem(R_\eps)\|_{C^2(\sphere^{n-1})}
=O(\eps).
\end{equation}
Since $f$ is Morse with exactly two critical points, $\sphere^{n-1}$ is a compact smooth manifold without boundary, and \eqref{eq7} shows that $g_\eps\to f$ in $C^2(\sphere^{n-1})$, Lemma~\ref{lem5} applies. Thus, $g_\eps$ has exactly two critical points for every sufficiently small $\eps>0$. The functions $g_\eps$ and $A_{K_\eps}$ differ only by a positive affine rescaling, so $A_{K_\eps}$ also has exactly two critical points.

For every $u\in\sphere^{n-1}$,
$A_{K_\eps}(-u)=\Haus^n(K_\eps)-A_{K_\eps}(u).$
Thus, the critical points of $A_{K_\eps}$ occur in antipodal pairs. Denote the two critical points by $u_\eps$ and $-u_\eps$. By Lemma~\ref{lemma1}, both correspond to a centroidal section, and they determine the same hyperplane $u_\eps^\perp$. Conversely, if an arbitrary affine hyperplane section of $K_\eps$ has centroid $o$, then the hyperplane itself contains $o$ and is therefore of the form $v^\perp$ for some $v\in\sphere^{n-1}$. Lemma~\ref{lemma1} then forces $v$ to be a critical point of $A_{K_\eps}$. Hence $K_\eps\cap u_\eps^\perp$ is the unique hyperplane section whose centroid agrees with $c(K_\eps)=o$.
\end{proof}

\vskip3pt 
{\bf Conflict of Interest}: We declare that we have no conflict of interest.

{\bf Data Availability}: Not applicable.

\end{document}